\documentclass[12pt]{article}
\usepackage{amsmath,amssymb,amsfonts}
\usepackage{epsfig}
\usepackage{tikz}
\usepackage{graphics}
\usepackage{bbm}
\usepackage{bm, bbm}
\usetikzlibrary[graphs, arrows, backgrounds, intersections, positioning, fit, petri, calc, shapes, decorations.pathmorphing]
\usepgflibrary{patterns}

\definecolor{darkred}{RGB}{105,0,0}

\newtheorem{theorem}{Theorem}

\newtheorem{proposition}[theorem]{Proposition}

\newtheorem{corollary}[theorem]{Corollary}
\newtheorem{conjecture}{Conjecture}

\newenvironment{proof}{\noindent\emph{Proof.}\hspace{.25em}}{\hspace*{\fill}
$\Box$\newline}
\def \QD1 {\hfill $\spadesuit$}

\newcommand{\De}{\Delta}
\newcommand{\de}{\delta}

\newcommand{\ganzo}{\mathbb{Z}_{\geq 0}}
\newcommand{\real}{\mathbb{R}}

\newcommand{\ganz}{\mathbb{Z}}

\numberwithin{equation}{section}

\begin{document}
\title{\bf Partitions of multigraphs under degree constraints}

\author{{{
Thomas Schweser}\thanks{partially supported by DAAD, Germany (as part of BMBF) and by the Ministry of Education Science, Research and Sport of the Slovak Republic within the project 57320575}
\thanks{
Technische Universit\"at Ilmenau, Inst. of Math., PF 100565, D-98684 Ilmenau, Germany. E-mail
address: thomas.schweser@tu-ilmenau.de}}
\and{{Michael Stiebitz}\footnotemark[1]
\thanks{
Technische Universit\"at Ilmenau, Inst. of Math., PF 100565, D-98684 Ilmenau, Germany. E-mail
address: michael.stiebitz@tu-ilmenau.de}}}

\date{}
\maketitle

\begin{abstract}
\noindent In 1996, Michael Stiebitz proved that if $G$ is a simple graph with $\de(G)\geq s+t+1$ and $s,t\in \ganzo$, then $V(G)$ can be partitioned into two sets $A$ and $B$ such that $\de(G[A])\geq s$ and $\de(G[B])\geq t$. In 2016, Amir Ban proved a similar result for weighted graphs. Let $G$ be a simple graph with at least two vertices, let $w:E(G) \to \real_{>0}$ be a weight function, let $s,t \in \real_{\geq 0}$, and let $W=\max_{e\in E(G)} w(e)$. If $\de(G)\geq s+t+2W$, then $V(G)$ can be partitioned into two sets $A$ and $B$ such that $\de(G[A])\geq s$ and $\de(G[B])\geq t$. This motivated us to consider this partition problem for multigraphs, or equivalently for weighted graphs $(G,w)$ with $w:E(G) \to \ganz_{\geq 1}$. We prove that if $s,t\in \ganz_{\geq 0}$ and $\de(G)\geq s+t+2W-1\geq 1$, then $V(G)$ can be partitioned into two sets $A$ and $B$ such that $\de(G[A])\geq s$ and $\de(G[B])\geq t$. We also prove a variable version of this result and show that for $K_4^-$-free graphs, the bound on the minimum degree can be decreased.
\end{abstract}



\section{Introduction and main results}
%{\bf \large A. History and Results}
%\medskip
All graphs considered in this paper are finite and undirected, and may have multiple edges but no loops. For a graph $G$, let $V(G)$ and $E(G)$ be the vertex set and the edge set of $G$, respectively. For $X,Y \subseteq V(G)$ let $E_G(X,Y)$ be the set of all edges joining a vertex of $X$ with a vertex of $Y$, and let $E_G(X) = E_G(X,X)$. The subgraph of $G$ {\em induced} by $X$
is denoted by $G[X]$, i.e., $V(G[X])=X$ and $E(G[X])=E_G(X)$. For a vertex $v$ of $G$, let $E_G(v)=E_G(\{v\},V(G) \setminus \{v\})$. Then $d_G(v)=|E_G(v)|$ is the {\em degree} of $v$ in $G$. As usual, $\de(G)$ is the {\em minimum degree} and $\De(G)$ is the {\em maximum degree} of $G$. Furthermore, we define the \emph{multiplicity} of two different vertices $u$ and $v$ by $\mu_G(u,v)=|E_G(u,v)|$. Given a vertex $u$, let $$w_G(u)=\max_{ v\in V(G) \setminus \{u\}}\mu_G(u,v)$$ be the \emph{weight} of $u$ in $G$. For the sake of readability, we will sometimes omit subscripts or brackets if the meaning is clear.

A sequence $(A_1,A_2, \ldots, A_p)$ of sets is called a {\em partition} of a set $V$ if $A_1, A_2, \ldots, A_p$ are pairwise disjoint non-empty subsets of $V$ such that their union is $V$.

Given a graph $G$ and two functions $a,b: V(G) \to \mathbb{R}_\geq 0$, a partition $(A,B)$ of $V(G)$ is called an {\em $(a,b)$-feasible partition} (of the vertex set) of $G$ if
\begin{itemize}
\item[{\rm (1)}] $d_{G[A]}(v)\geq a(v)$ for all $v\in A$, and
\item[{\rm (2)}] $d_{G[B]}(v)\geq b(v)$ for all $v\in B$.
\end{itemize}

In 1996, Stiebitz \cite{Stieb} proved the following partition result for  simple graphs, thus solving a conjecture due to Thomassen \cite{Thomassen}. Let $G$ be a simple graph, and let $a,b: V(G) \to \mathbb{Z}_{\geq 0}$ be two functions. Assume that $$d_G(v) \geq a(v) + b(v) + 1$$ for every vertex $v \in V(G)$. Then, there is an $(a,b)$-feasible partition of $G$. With a similar approach, Ban \cite{Ban} proved a related result for weighted graphs. Let $G$ be a simple graph, and let $w:E(G)\to
\mathbb{R}_{> 0}$ be a {\em weight function} for $G$. Then,
$$d_G(v)=\sum_{e \in E_G(v)} w(e) \mbox{ and } W_G(v)= \max_{e \in E_G(v)} w(e)$$
is the {\em weighted degree} of $v$ in $G$, respectively the {\em maximum weight} of an edge incident to $v$ in $G$. Moreover, let $a,b:V(G) \to \mathbb{R}_{\geq 0}$ be two functions. Assume  that $d_G(v) \geq a(v) + b(v) + 2W_G(v)$ for every vertex $v \in V(G)$. Then, there is an $(a,b)$-feasible partition of $G$. As Ban noticed, the boundary is sharp.

If we set $w(e)=1$ for all $e \in E(G)$ and assume that $a,b: V(G) \to \mathbb{Z}_{\geq 0}$, then Ban's Theorem states the same as Stiebitz's Theorem, except that Ban requires $d_G(v) \geq a(v) + b(v) + 2$ for all $v\in V(G)$, whereas Stiebitz only needs $d_G(v) \geq a(v) + b(v) + 1$ for all $v\in V(G)$. Hence, one may wonder, if at least in the case of a graph with integer weights, i.e. in the case of a multigraph, $d_G(v) \geq a(v) + b(v) + 2w_G(v) - 1$ for all $v \in V(G)$ is sufficient for the existence of an $(a,b)$-partition of $G$. That this is indeed the case, states the following theorem.

\begin{theorem}
\label{Theorem:Hauptsatz}
Let $G$ be a graph with $\de(G)\geq 1$, and let $a,b: V(G) \to \mathbb{Z}_{\geq0}$ be two functions such that $d_G(v) \geq a(v) + b(v) + 2w_G(v) - 1$ for every $v \in V(G)$. Then, there is an  $(a,b)$-feasible partition of $G$.
\end{theorem}

Kaneko \cite{Kan} as well as Bazgan, Tuza and Vanderpooten \cite{BazganTV} examined the case, that $G$ is a triangle-free simple graph. While Kaneko only considered constant functions $a$ and $b$, Bazgan, Tuza and Vanderpooten generalized Kaneko's result to the case of variable functions. Let $G$ be a triangle-free simple graph, and let $a,b:V(G) \to \mathbb{Z}_{\geq 1}$ be two functions. Assume, that $d_G(v) \geq a(v) + b(v)$ for all  $v \in V(G)$. Then, $G$ has an $(a,b)$-feasible partition. This theorem can also be extended to multigraphs.

\begin{theorem}
\label{Theorem:Dreikreisfrei}
Let $G$ be a triangle-free graph with $\de(G)\geq 1$, and let $a,b: V(G) \to \mathbb{Z}_{\geq 1}$ be two functions such that $d_G(v) \geq a(v) + b(v) + 2w_G(v) - 2$ for every $v \in V(G)$. Then, there is an $(a,b)$-feasible partition of $G$.
\end{theorem}

In 2017, Liu and Xu \cite{LiuXu} proved that one may obtain the same boundary as in the triangle-free case for $K_4^-$-free graphs, whereby $K_4^-$ denotes the graph that results from the $K_4$ by removing one edge. Let $H$ and $G$ be graphs. We say that $G$ is $H$\emph{-free}, if $G$ does not contain a graph isomorphic to $H$ as a subgraph. For graphs with multiple edges, we obtain the following.

\begin{theorem}
\label{Theorem:K_4^--free}
Let $G$ be a $K_4^-$-free graph with $\delta(G) \geq 1$, and let $a,b : V(G) \to \mathbb{Z}_{\geq 1}$ be two functions such that $d_G(v) \geq a(v) + b(v) + 2w_G(v) - 2$ for every $v \in V(G)$. Then, there is an $(a,b)$-feasible partition of $G$.
\end{theorem}

Note that this theorem obviously implies Theorem~\ref{Theorem:Dreikreisfrei}, so we abstain from giving an extra proof of the triangle-free case. By considering graphs with girth at least $5$, Diwan \cite{Diwan} as well as Gerber and Kobler \cite{GerberKo} managed to soften the degree-condition even more for constant, respectively variable functions $a,b$. Let $G$ be a simple graph with girth at least $5$, and let $a,b: V(G) \to \mathbb{Z}_{\geq 2}$ be two functions. Assume that $d_G(v) \geq a(v) + b(v) - 1$ for all $v \in V(G)$. Then, $G$ has an $(a,b)$-feasible partition.

Liu and Xu also generalized this theorem by considering triangle-free graphs in which each vertex is contained in at most one cycle of length $4$. They obtained the following. Let $G$ be a triangle-free simple graph of which each vertex is contained in at most one cycle of length $4$. Moreover, let $a,b: V(G) \to \mathbb{Z}_{\geq 2}$ be two functions such that $d_G(v) \geq a(v) + b(v) - 1$ for all $v \in V(G)$. Then, $G$ has an $(a,b)$-feasible partition.

It shows that it is not easily possible to adjust the proof neither of Gerber and Kobler's nor of Liu and Xu's Theorem in order to obtain the related statements for graphs in general. Our proof attempt in the first version of this paper turned out wrong. However, the evidence available indicates that the related statement is true, so we feel confident to conjecture the following.

\begin{conjecture}
\label{Conjecture:Girth5}
Let $G$ be a graph not containing cycles of length $3$ and $4$, and let $a,b: V(G) \to \mathbb{Z}_{\geq 2}$ be two functions such that $d_G(v) \geq a(v) + b(v) + 2w_G(v) - 3$ for every $v \in V(G)$. Then, there is an $(a,b)$-feasible partition of $G$.
\end{conjecture}

\begin{conjecture}
Let $G$ be a triangle-free graph of which each vertex is contained in at most one cycle of length $4$. Moreover, let $a,b: V(G) \to \mathbb{Z}_{\geq 2}$ be two functions such that $d_G(v) \geq a(v) + b(v) + 2w_G(v) - 3$ for every $v \in V(G)$. Then, there is an $(a,b)$-feasible partition of $G$.
\end{conjecture}

\section{Preliminary considerations}

In this section we shall prove two auxiliary results. First we need some notation. Let $G$ be a graph, and let $f : V(G) \to \mathbb{Z}_{\geq 0}$ be a function. If $X\subseteq V(G)$ and $v \in V(G)$, then we write $d_X(v)$ for the degree of $v$ in $G[X \cup \{v\}]$ (note that $v\in X$ may or may not hold). A subset $X \subseteq V(G)$ is said to be {\em $f$-meager} in $G$, if for every non-empty induced subset $Y$ of $X$ there is a vertex $v \in Y$ such that $d_Y(v) \leq f(v) + w_G(v) - 1$. A set $X \subseteq V(G)$ is called {\em $f$-nice} in $G$, if $d_X(v) \geq f(v)$ for all $v \in X$.

Now let $a,b: V(G) \to \mathbb{R}_\geq 0$ be two functions. Note that a partition $(A,B)$ is an $(a,b)$-feasible partition of $G$ if and only if $A$ is $a$-nice in $G$ and $B$ is $b$-nice in $G$. We say that a pair $(A,B)$ is an {\em $(a,b)$-feasible pair} of $G$, if $A$ and $B$ are disjoint subsets of $V(G)$ such that $A$ is $a$-nice and $B$ is $b$-nice. Furthermore, a partition $(A,B)$ of $V(G)$ is called an {\em $(a,b)$-meager} partition of $G$, if $A$ is $a$-meager in $G$ and $B$ is $b$-meager in $G$.

In the proofs of Theorem \ref{Theorem:Hauptsatz} and Theorem \ref{Theorem:K_4^--free} we will need the following two observations related to \cite{Stieb}.

\begin{proposition}
\label{FeasPart}
Let $G$ be a graph with $\delta(G) \geq 1$, and let $a,b: V(G) \to \mathbb{Z}_{\geq 0}$ be two functions such that $d_G(v) \geq a(v) + b(v) + 2w_G(v) - 3$ for all $v \in V(G)$. If there exists an $(a,b)$-feasible pair, then there exists an $(a,b)$-feasible partition of $G$, too.
\end{proposition}
\begin{proof}
Note that $\delta(G) \geq 1$ implies $w_G(v)\geq 1$ for all $v \in V(G)$. Consider an $(a,b)$-feasible pair $(A,B)$ such that $A \cup B$ is maximal. All we need to show is that $A \cup B = V(G)$. Assume, in contrary, that $C = V(G) \setminus (A \cup B)$ is non-empty. Since $A \cup B$ is maximal, this implies that $(A, B \cup C)$ is not $(a,b)$-feasible. Thus, there exists a vertex $v \in C$ such that $d_{B \cup C}(v) \leq b(v) - 1$. Due to the fact that $d_G(v) \geq a(v) + b(v) + 2w_G(v) - 3$, we conclude $d_A(v) \geq a(v) + 2w_G(v) - 2 \geq a(v)$. But then $(A \cup \{v\}, B)$ is an $(a,b)$-feasible pair, in contradiction to the maximality of $A \cup B$.
\end{proof}

\begin{proposition}
\label{Non-Meager}
Let $G$ be a graph with $\delta(G) \geq 1$, and let $a,b : V(G) \to \mathbb{Z}_{\geq 1}$ be two functions. Moreover assume, that
$$d_G(v) \geq a(v) + b(v) + 2w_G(v) - 2$$
for all $v \in V(G)$.
Then, $G$ has an $(a,b)$-feasible partition, provided that there is no $(a,b)$-meager partition of $G$.
\end{proposition}
\begin{proof}
Since $\delta(G) \geq 1$, we have $|G| \geq 2$ and $w_G(v) \geq 1$ for all $v \in V(G)$. We choose a non-empty subset $A$ of $V(G)$ such that
\begin{itemize}
\item[(a)] $A$ is $a$-nice, and
\item[(b)] $|A|$ is minimum subject to (a).
\end{itemize}
Let $B=V(G) \setminus A$. Obviously, $V(G) \setminus \{v\}$ satisfies (a) for each vertex $v$, therefore, $A$ exists and $B$ is non-empty. Because of (b), for every proper subset $A'$ of $A$ that is non-empty, we find a vertex $u \in A'$ fulfilling $d_{A'}(u) \leq a(u) - 1$. Since $A'$ can be chosen such that $|A'|=|A|-1$, this implies the existence of a vertex $u \in A$ such that $d_A(u) \leq a(u) + w_G(u) - 1$. Thus, $A$ is $a$-meager. Since there is no $(a,b)$-meager partition of $G$, $B$ is not $b$-meager and, therefore, there is a non-empty subset $B'$ of $B$ such that $d_{B'}(v) \geq b(v) + w_G(v) \geq b(v)$ for all $v \in B'$. Hence, $(A,B')$ is an $(a,b)$-feasible pair and Proposition~\ref{FeasPart} implies the existence of an $(a,b)$-feasible partition of $G$.
\end{proof}

Let $G$ be a graph, let $a, b:V(G) \to \ganzo$ be two functions, and let $(A,B)$ be a partition of $V(G)$. We define the {\em $(a,b)$-weight} $w(A,B)$ as
$$w(A,B) = |E_G(A)| + |E_G(B)| + \sum_{v \in A} b(v) + \sum_{v \in B} a(v).$$
If $|A| \geq 2$ we can choose an arbitrary vertex $x \in A$ and $(A - \{x\}, B \cup \{x\})$ remains a partition of $V(G)$. In particular, it holds
\begin{equation}
\label{eq:AustauschA-x}
w(A - \{x\}, B \cup \{x\}) - w(A,B)  =  d_B(x) - d_A(x) + a(x) - b(x).
\end{equation}
%
%\begin{IEEEeqnarray}{rcl}\label{eq:AustauschA-x}
%w(A - \{x\}, B \cup \{x\}) - w(A,B) & = & d_B(x) - d_A(x) + a(x) - b(x).
%\end{IEEEeqnarray}
Similarly, if $|B| \geq 2$ we may choose $y \in B$ and $(A \cup \{y\}, B - \{y\})$ is also a partition of $V(G)$ fulfilling
\begin{equation}
\label{eq:AustauschB-y}
w(A \cup \{y\}, B - \{y\}) - w(A,B) = d_A(y) - d_B(y) + b(y) - a(y).
\end{equation}
If $v \in A$, then $d_{A \setminus \{v\}}(u) = d_A(u) - \mu(u,v)$ for all $u \in A \setminus \{v\}$ and, hence,
\begin{equation}
\label{eq:vertex_deletion}
d_{A \setminus \{v\}}(u) \geq d_A(u) - w_G(u) \quad \text{and} \quad d_{A \setminus \{v\}} \geq d_A(u) - w_G(v).
\end{equation}
%
%\begin{IEEEeqnarray}{rcl}\label{eq:AustauschB-y}
%w(A \cup \{y\}, B - \{y\}) - w(A,B) &=& d_A(y) - d_B(y) + b(y) - a(y).
%\end{IEEEeqnarray}

\section{Proof of Theorem \ref{Theorem:Hauptsatz}}

Let $G$ be a graph with $\delta(G)\geq 1$, and let $a,b: V(G) \to \mathbb{Z}_{\geq 0}$ be two functions such that
\begin{equation} \label{eq:minimum_degree_constraint}
d_G(v) \geq a(v) + b(v) + 2w_G(v) - 1
\end{equation}
for all $v \in V(G)$.
Our aim is, to prove that there is an $(a,b)$-feasible partition of $G$.
Since $\delta(G)\geq 1$, we obtain that $w_G(v) \geq 1$ for all $v \in V(G)$.
As a conclusion, if there exists a vertex $v$ such that $a(v)=0$ or $b(v)=0$, then equation \eqref{eq:minimum_degree_constraint} implies that $(\{v\}, V(G) \setminus \{v\})$ or $(V(G)\setminus \{v\}, \{v\})$, respectively, is an $(a,b)$-feasible partition of $G$. Thus, in the following we may assume that
\begin{equation} \label{eq:a(v)&b(v)geq1}
a(v) \geq 1 \text{ and } b(v) \geq 1
\end{equation}
for every vertex $v \in V(G)$.
If there is no $(a,b)$-meager partition of $V(G)$, then we conclude from
Proposition~\ref{Non-Meager} that there is an $(a,b)$-feasible partition of $G$, and we are done. It remains to consider the case that there is an  $(a,b)$-meager partition of $G$. Then, we choose an $(a,b)$-meager partition $(A,B)$ of $G$ such that $w(A,B)$ is maximum. Since $A$ is $a$-meager, there is a vertex $x \in A$ fulfilling $d_A(x) \leq a(x) + w_G(x) - 1$. Together with equations \eqref{eq:minimum_degree_constraint} and \eqref{eq:a(v)&b(v)geq1}, this implies $d_B(x) \geq b(x) + w_G(x) \geq w_G(x) + 1$. Hence, $|B| \geq 2$. By symmetry, we conclude that $|A| \geq 2$, too.

Next, we claim the existence of a non-empty subset $\tilde{A} \subseteq A$ such that $d_{\tilde{A}}(v) \geq a(v)$ for all $v \in \tilde{A}$. Otherwise, $A$ is $(a-1)$-degenerate and, by \eqref{eq:vertex_deletion}, we conclude that $A \cup \{y\}$ is $a$-meager for all $y \in B$. Since $B$ is $b$-meager, there exists a vertex $y \in B$ such that $d_B(y) \leq b(y) + w_G(y) - 1$ and, by \eqref{eq:minimum_degree_constraint}, we have
$d_A(y) \geq a(y) + w_G(y)$. Because of $|B| \geq 2$, the pair $(A \cup \{y\},B \setminus\{y\})$ is an $(a,b)$-meager partition of $G$ such that
$$w(A \cup \{y\},B \setminus \{y\}) - w(A,B) = d_{A}(y) - d_B(y) + b(y) - a(y) \geq 1$$ (by \eqref{eq:AustauschB-y}), in contradiction to the maximality of $w(A,B)$. This proves the claim. By symmetry, there is also a non-empty subset $\tilde{B} \subseteq B$ such that $d_{\tilde{B}} \geq b(v)$ for all $v \in \tilde{B}$ and, hence, $(\tilde{A}, \tilde{B})$ is an $(a,b)$-feasible pair. By Proposition~\ref{FeasPart}, this implies the existence of an $(a,b)$-feasible partition of $G$, and the proof is complete.

\section{Proof of Theorem~\ref{Theorem:K_4^--free}}
Let $G$ be a $K_4^-$-free graph with $\delta(G) \geq 1$, and let $a,b: V(G) \to \mathbb{Z}_{\geq 1}$ be two functions such that
\begin{equation}\label{eq:degree_constraint_K_4^--free}
d_G(v) \geq a(v) + b(v) + 2w_G(v) - 2
\end{equation}
for all $v \in V(G)$. If $G$ has an $(a,b)$-feasible pair, then we are done by Proposition~\ref{FeasPart}. Otherwise, $G$ contains no $(a,b)$-feasible pair. Firstly, we claim that $|G| \geq 3$. Otherwise, $|G|=2$ and, thus, $V(G)=\{u,v\}$.  Since $\delta(G) \geq 1$, equation \eqref{eq:degree_constraint_K_4^--free} implies $$d_G(v) \geq a(v) + b(v) + 2w_G(v) - 2 \geq 2w_G(v)=2d_G(v),$$ which is not possible. As a consequence, $|G| \geq 3$. Secondly, we show that for each edge $uv \in E(G)$ it holds
\begin{equation}\label{eq:K_4^--free:a(x)+a(y)geq3}
a(u) + a(v) \geq 3 \quad \text{and} \quad b(u) + b(v) \geq 3.
\end{equation}
Assume, to the contrary, that there is an edge $uv \in E(G)$ that does not fulfill the above equation. By symmetry, we can assume that $b(u)=b(v)=1$. Together with equation \eqref{eq:degree_constraint_K_4^--free}, this implies that $d_G(u) \geq a(u) + w_G(u)$ and that $d_G(v) \geq a(v) + w_G(v)$. If $N_G(u) \cap N_G(v)$ is empty, we claim that for $A=V(G) \setminus \{u,v\}$ and $B= \{u,v\}$, the pair $(A,B)$ is an $(a,b)$-feasible partition. This follows from the fact that $$d_{A}(w) \geq d_G(w) - w_G(w) \geq a(w)$$ for all $w \in A$ (by \eqref{eq:degree_constraint_K_4^--free}), $d_B(w) \geq 1 = b(w)$ for all $w \in B$, and that $A$ is non-empty because of $|G| \geq 3$. Hence, $N_G(u) \cap N_G(v)$ is non-empty. Due to the fact that $G$ is $K_4^-$-free, $N_G(u) \cap N_G(v)$ consists of exactly one vertex $w$ and each vertex of $V(G) \setminus \{u,v,w\}$ has at most one neighbor within $\{u,v,w\}$. If $b(w)=1$, then $(V(G) \setminus \{u,v,w\},\{u,v,w\})$ is an $(a,b)$-feasible partition of $V(G)$ (by a similar argument as before). Otherwise, $b(w) \geq 2$ and, by \eqref{eq:degree_constraint_K_4^--free}, $$d_G(w) \geq a(w) + 2w_G(w).$$ However, this leads to $(V(G) \setminus \{u,v\},\{u,v\})$ being an $(a,b)$-feasible partition, a contradiction. This proves the claim that \eqref{eq:K_4^--free:a(x)+a(y)geq3} holds.

If $G$ has no $(a,b)$-meager partition, then we are done by Proposition~\ref{Non-Meager}.
%we choose an $a$-nice subset $A$ of $V(G)$ with minimum size. Let $B = V(G) \setminus A$.  Since $V \setminus \{w\}$ is $a$-nice for all $w \in V(G)$, $B$ is non-empty. By construction, $A$ is $a$-degenerate and, more particularly, $a$-meager. Thus, $B$ is not $b$-meager and hence there is a non-empty subset $B'$ of $B$ such that $d_{B'}(v) > b(v) + w_G(v) - 1 \geq b(v)$. Consequently, $(A,B')$ is an $(a,b)$-feasible pair and we are done.
Hence, we only need to consider the case that there is an $(a,b)$-meager partition of $G$, say $(A,B)$. Since $G$ has no $(a,b)$-feasible pair and by symmetry, we can assume that $B$ is $(b-1)$-degenerate and, therefore, $(b-1)$-meager. We choose an $(a,b-1)$-meager partition $(A,B)$ such that
\begin{itemize}
\item[(1)] $w(A,B)$ is maximum and
\item[(2)] $|A|$ is minimum subject to (1).
\end{itemize}
Firstly, we show that $B$ contains at least two vertices. Otherwise, $B=\{y\}$. Since $A$ is $a$-meager, there is an $x \in A$ such that $d_A(x) \leq a(x) + w_G(x) - 1$, and, by \eqref{eq:degree_constraint_K_4^--free}, $$d_B(x) \geq b(x) + w_G(x) - 1 \geq w_G(x).$$ Since $B$ consists of only one vertex, this implies that $d_B(x)=w_G(x)$ and, hence, $b(x)=1$. By \eqref{eq:K_4^--free:a(x)+a(y)geq3}, $b(y) \geq 2$ and, thus, $$d_{B \cup \{x\}}(y)=w_G(x) = 2 + w_G(x) - 2 \leq b(y) + w_G(y) - 2.$$ Since $|G| \geq 3$, this implies that $(A \setminus \{x\}, B \cup \{x\})$ is also an $(a,b-1)$-meager partition and, by \eqref{eq:AustauschA-x},
\begin{equation*}
w(A\setminus\{x\}, B \cup \{x\}) - w(A,B)  =  d_B(x) - d_A(x) + a(x) - b(x) \geq 0,
\end{equation*}
contradicing the choice of $(A,B)$. Hence, $|B| \geq 2$.

Since $B$ is $(b-1)$-meager we can choose a $y \in B$ such that $d_B(y) \leq b(y) + w_G(y) - 2$. Now we claim that there are $x,z \in A$ fulfilling $d_A(x) \leq a(x) + w_G(x) - 1$ and $d_A(z) \leq a(z) + w_G(z) - 1$ such that $G[\{x,y,z\}]$ contains a triangle. Due to the fact that $d_B(y) \leq b(y) + w_G(y) - 2$, we get
$$d_A(y)\geq a(y) + w_G(y) \geq w_G(y) + 1,$$
implying that $|A| \geq 2$. Let $A' = A \cup \{y\}$ and let $B' = B \setminus \{y\}$.
By \eqref{eq:AustauschB-y},
\begin{eqnarray*}
w(A',B') - w(A,B) & = & d_A(y) - d_B(y) + b(y) - a(y)\\
& \geq & a(y) + w_G(y) - b(y) - w_G(y) + 2 + b(y) - a(y) \\
& \geq & 2.
\end{eqnarray*}
Since $B'$ is still $(b-1)$-meager, this implies that $A'$ is not $a$-meager since otherwise this would contradict the choice of $(A,B)$. Thus, there exists a subset $\tilde{A}$ of $A$ such that $\tilde{A} \cup \{y\}$ is $(a + w_G)$-nice and, therefore, $\tilde{A}$ is $a$-nice. If $A$ is not $a$-nice, then there is an $x' \in A \setminus \tilde{A}$ such that $d_A(x') \leq a(x') - 1$ and, by \eqref{eq:degree_constraint_K_4^--free}, $d_B(x') \geq b(x') + 2w_G(x') - 1$. Let $A'' = A \setminus \{x'\}$ and $B'' = B \cup \{x'\}$. Then, $A''$ is still $a$-meager and, since $\tilde{A} \subseteq A''$ and due to the fact that $G$ contains no $(a,b)$-feasible pair, $B''$ is $(b-1)$-degenerate. Hence, $(A'',B'')$ is an $(a,b-1)$-meager partition and, by \eqref{eq:AustauschA-x},
\begin{eqnarray*}
w(A'',B'') - w(A,B)& = & d_B(x') - d_A(x') + a(x') - b(x') \\
& \geq & b(x') + 2w_G(x') - 1 - a(x') + 1 + a(x') - b(x') \\
& \geq & 2,
\end{eqnarray*}
contradicting the choice of $(A,B)$. Consequently, $A$ is $a$-nice. Let $$C = \{ u \in A ~|~ d_A(u) \leq a(u) + w_G(u) - 1\}.$$
Since $A$ is $a$-meager, $C$ is non-empty. We claim that $C \subseteq \tilde{A}$. Otherwise, there is an $x' \in C \setminus \tilde{A}$. Then, by \eqref{eq:degree_constraint_K_4^--free}, $d_B(x') \geq b(x')+ w_G(x') - 1 \geq b(x')$. Let again $A''=A \setminus \{x'\}$, and $B''= B \cup \{x'\}$. Since $\tilde{A} \subseteq A''$ and $\tilde{A}$ is $a$-nice, the set $B''$ must be $(b-1)$-meager and, hence, $(A'',B'')$ is an $(a,b-1)$-meager partition. Furthermore, by \eqref{eq:AustauschA-x}, it follows that
\begin{eqnarray*}
w(A'',B'')-w(A,B) & = & d_B(x') - d_A(x') + a(x') - b(x') \geq 0.
\end{eqnarray*}
Since $|A''| < |A|$, this contradicts the choice of $(A,B)$. Hence, the claim that $C \subseteq \tilde{A}$ is proven. Since $\tilde{A} \cup \{y\}$ is $(a+w_G)$-nice, it follows that $C \subseteq N_G(y)$, which leads to $C \subseteq \tilde{A} \cap N_G(y)$.

Let $x \in C$. If $N_G(x) \cap C = \varnothing$, then let $A'' = A \setminus \{x\}$ and $B''=B \cup \{x\}$. Since $A$ is $a$-nice, and since $d_G(v) \geq a(v) + w_G(v)$ for all $v \in A \setminus C$, $A''$ is still $a$-nice and, thus, $B''$ is $(b-1)$-meager. However, \eqref{eq:AustauschA-x} implies
\begin{eqnarray*}
w(A'',B'')-w(A,B) & = & d_B(x) - d_A(x) + a(x) - b(x) \\
& \geq & b(x) +w_G(x) - 1 - a(x)\\
&& - w_G(x) + 1 + a(x) - b(x)\\
& \geq &0.
\end{eqnarray*}
Since $|A''| < |A|$, this contradicts the choice of $(A,B)$. Hence, there is a vertex $z \in N_G(x) \cap C$ and $G[\{x,y,z\}]$ contains a triangle, as claimed.
To complete the proof, let $A''= A \setminus \{x,z\}$ and let $B''= B \cup \{x,z\}$. We claim that $(A'',B'')$ is an $(a,b-1)$-meager partition. Firstly, we show that $A'' \neq \varnothing$. Otherwise, since $\{x,z\} \subseteq C \subseteq \tilde{A} \subseteq A$, we have $C=\tilde{A}=A$. Thus, $A'=A \cup \{y\} = \{x,y,z\}$ is $(a + w_G)$-nice. Since $G$ contains no $K_4^-$, each vertex of $B'= B \setminus \{y\}$ has at most one neighbor within $\{x,y,z\}$.
Thus, $d_{B'}(v) \geq d_G(v) - w_G(v)$ for all $v \in B'$ and, by \eqref{eq:degree_constraint_K_4^--free}, we get
$$d_{B'}(v) \geq  a(v) + b(v) + w_G(v) - 2 \geq b(v)$$
for all $v  \in B'$, which implies that $(A',B')$ is an $(a,b)$-feasible partition of $V(G)$, a contradiction. As a consequence, $A'' \neq \varnothing$. Since $G$ contains no $K_4^-$ and since $C \subseteq N_G(y)$, $N_G(x) \cap C=\{z\}$ and $N_G(z)\cap C = \{x\}$. Thus, if $v \in A'' \cap C$, then $v$ has no neighbour in $\{x,z\}$, and therefore $d_{A''}(v) \geq a(v)$ (since $A$ is $a$-nice). If $v \in A'' \setminus C$, then $v$ has at most one neighbor in $\{x,z\}$. Thus, we get $$d_{A''}(v) \geq d_A(v) - w_G(v) \geq a(v) + w_G(v) - w_G(v) \geq a(v).$$  As a consequence, $A''$ is $a$-nice (and non-empty). Since $G$ has no $(a,b)$-feasible partition, this implies that $B''$ is $(b-1)$-meager and $(A'',B'')$ is an $(a,b-1)$-meager partition. On the other hand,
\begin{eqnarray*}
w(A'',B'') - w(A,B) & \geq & d_B(x) + d_B(z) - d_A(x) - d_A(z) + 2 \\
& & - b(x) - b(z) + a(x) + a(z) \\
& \geq & b(x) + w_G(x) - 1 + b(z) + w_G(z) - 1 \\
& & - a(x) - w_G(x) + 1 - a(z) - w_G(z) + 1 + 2 \\
& & - b(x) - b(z) + a(x) + a(x) \geq  2,
\end{eqnarray*}
which contradicts the choice of $(A,B)$. This completes the proof.

\section{Concluding remarks}

It follows from a simple induction that Theorems \ref{Theorem:Hauptsatz} and \ref{Theorem:K_4^--free} can be extended to partitions of more than two sets.

\begin{corollary}
Let $G$ be a graph, and let $f_1, f_2, \ldots, f_p:V(G) \to \ganz_{\geq h-1}$ be $p$ functions with $p\geq 2$ and $h\in \{1,2\}$. Assume that $\de(G)\geq 1$ and
$$d_G(v)\geq f_1(v)+f_2(v)+\cdots + f_p(v)+(p-1)(2w_G(v)- h)$$
for all $v\in V(G)$. Then, there is a partition $(A_1,A_2, \ldots, A_p)$ of $V(G)$ such that $d_{A_i}(v)\geq f_i(v)$ for every $i\in \{1,2, \ldots, p\}$ and every $v\in A_i$, provided that either $h=1$, or $h=2$ and $G$ is $K_4^-$-free.
\end{corollary}

If we renounce the condition $\delta(G) \geq 1$ in Theorem \ref{Theorem:Hauptsatz}, it may happen that $G$ only consists of isolated vertices. But then, if $a(v)=1$ for all $v \in V(G)$ and if $b(v)=0$ for all $v \in V(G)$, the only possible choice would be $A=\emptyset$ and $B= V(G)$, which is not a partition. However, demanding $G$ to have at least one edge is sufficient. In this case we can delete all isolated vertices and Theorem \ref{Theorem:Hauptsatz} implies the existence of an $(a,b)$-feasible partition for the remaining graph. By inserting each isolated vertex $v$ to the set $A$ if $a(v) = 0$ and to $B$ if $b(v)=0$ will do the trick then. For Theorem \ref{Theorem:Dreikreisfrei} it is obvious that we cannot give up the condition $\delta(G) \geq 1$.

As an answer of Thomassen's Conjecture \cite{Thomassen}, we obtain the following for graphs in general.

\begin{corollary}
\label{Corollary:constant_functions}
Let $G$ be a graph and let $s,t \geq 0$ be integers. Assume that $\delta(G) \geq s + t + 2 \mu(G) - 1 \geq 1$. Then, $V(G)$ can be partitioned into two sets $A,B$ such that $\delta(G[A]) \geq s$ and $\delta(G[B]) \geq t$.
\end{corollary}

The next question standing to reason is as follows. Is it really necessary to require $d_G(v) \geq a(v) + b(v) + 2w_G(v) - 1$ in Theorem~\ref{Theorem:Hauptsatz}, or can we maybe save another value one. To answer this, consider the case $G=tK_3$, that is, $G$ is the graph on $3$ vertices in which each two vertices are joined by $t$ edges. By requesting only $d_G(v) \geq a(v) + b(v) + 2t - 2$, setting $a(v)=1$ and $b(v)=1$ for all $v \in V(G)$ leads to a counter-example, since obviously there is no partition $(A,B)$ such that $d_A(v) \geq 1$ for all $v \in A$ and $d_B(v) \geq 1$ for all $v \in B$, although $d_G(v)=2t$ is fulfilled.

For another counter-example consider the graph $H$ shown in Figure \ref{pic:cube}. If we set $G=tH$ for some $t \geq 1$, it holds $4t=d_G(v)$ for all $v \in V(G)$. By requiering only $d_G(v) \geq a(v) +  b(v) + 2t - 2$, setting $a(v)=1$ for all $v \in V(G)$ and $b(v) = 2t + 1$ gives us a counter-example. This is due to the fact that if $(A,B)$ is an $(a,b)$-feasible partition of $G$, then $A$ consists of at least two adjacent vertices that have a common neighbour $v$ in $B$. But then, $d_B(v) \leq 2t$, a contradiction.
If $H$ is the icosahedron and if $G=tH$ for some $t \geq 1$, then, by the same argumentation as above, setting $a(v)=1$ and $b(v)=3t+1$ leads to a third counter-example. Note that these examples prove that the boundary in Corollary~\ref{Corollary:constant_functions} is sharp, as well.

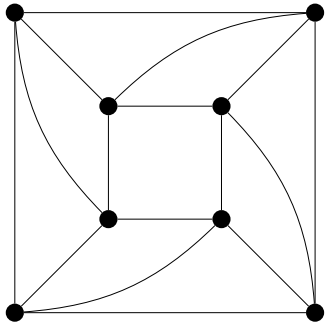
\begin{figure}[thbp]
\centering
\resizebox{.33\linewidth}{!}{
%\documentclass{scrreprt}
%\usepackage[ngerman]{babel}
%\usepackage{fontspec}
%\usepackage{tikz}
%\usepackage{graphics}
%\usepackage{bm, bbm}
%\usetikzlibrary[graphs, arrows, backgrounds, intersections, positioning, fit, petri, calc, shapes, decorations.pathmorphing, shapes.geometric]
%\begin{document}
%\tikzset{blue help lines/.style={help lines,color=blue}, red help lines/.style={help lines, color=red , >=stealth', shorten <=1.5pt, shorten >=1.5pt}}
%\colorlet{labelcolor}{black!50!blue}
%\colorlet{darkred}{black!50!red}
%\colorlet{darkblue}{black!50!blue}
%\begin{figure}

\begin{tikzpicture} [node distance=1cm, bend angle=20,
vertex/.style={circle,minimum size=3mm,very thick, draw=black, fill=black, inner sep=0mm}, information text/.style={fill=red!10,inner sep=1ex, font=\Large}, help lines/.style={-,color=black}]

\node[draw=none,minimum size=3cm,regular polygon,regular polygon sides=4] (a) {};
\foreach \x in {1,2,3,4}
\node[vertex] (v\x) at (a.corner \x) {};

\node[draw=none,minimum size=8cm,regular polygon,regular polygon sides=4] (b) {};
\foreach \x in {1,2,3,4}
\node[vertex] (u\x) at (b.corner \x) {};

\path[-]
		(v1)
		edge[help lines] (v2)
		edge[help lines] (u1)
		edge[help lines, bend left] (u4)
		(v2)
		edge[help lines, bend left] (u1)
		edge[help lines] (u2)
		edge[help lines] (v3)
		(v3)
		edge[help lines, bend left] (u2)
		edge[help lines] (u3)
		edge[help lines] (v4)
		(v4)
		edge[help lines, bend left] (u3)
		edge[help lines] (u4)
		edge[help lines] (v1)
		(u1)
		edge[help lines] (u2)
		(u2)
		edge[help lines] (u3)
		(u3)
		edge[help lines] (u4)
		(u4)
		edge[help lines] (u1);
%		edge[help lines] (w4)
%		(w4)
%		edge[help lines] (u2)
%		edge[help lines] (u3)
%		edge[help lines] (w5)
%		(w5)
%		edge[help lines] (u3)
%		edge[help lines] (w6)
%		(w6)
%		edge[help lines] (u3)
%		edge[help lines] (u1)
%		(u1)
%		edge[help lines] (u2)
%		edge[help lines] (u3)
%		(u2)
%		edge[help lines] (u3);
%		
%	

%		(v1)
%		edge[help lines] node[name=h1, midway, sloped, above=1pt] {} (v2)
%		edge[help lines] node[name=h2, midway, sloped, above=1pt] {} (v5)
%		edge[help lines] node[name=h3, midway, sloped, above=1pt] {} (v6)
%		edge[help lines] node[name=h4, midway, sloped, above=1pt] {} (v7)
%		(v2)
%		edge[help lines] node[name=h5, midway, sloped, above=1pt] {} (v6)
%		
%		(v3)
%		edge[help lines] node[name=h6, midway, sloped, below=2pt] {} (v2)
%		edge[help lines] node[name=h7, midway, sloped, above=2pt] {} (v4)
%		edge[help lines] node[name=h8, midway, sloped, below=2pt] {} (v6)
%		(v4)
%		edge[help lines] node[name=h9, midway, sloped, below=2pt] {} (v5)
%		edge[help lines] node[name=h10, midway, sloped, below=2pt] {} (v7)
%		(v5)
%		edge[help lines] node[name=h11, midway, sloped, above=1pt] {} (v7);
%
%\node at (h7) [yshift=-1cm]{$V$};	
%
%\node at (h7) [yshift=-1cm]{$G^\circ=(V,E)$};
%	
%\foreach \x in {1,2,...,5,7,11}
%\node [color=red] at (h\x) {$\bm{+}$};
%
%\foreach \x in {3,4,6,8,9,10}
%\node [color=red] at (h\x) {$\bm{-}$};
%\node at (h7) [yshift=-1cm]{$G=(V,E, {\color{red}\sigma})$};
\end{tikzpicture}

%\end{figure}
%\end{document}
}
\caption{The graph $H$ serves as a starting point to design a counter-example.}
\label{pic:cube}
\end{figure}

\end{document}